\documentclass[12pt]{amsart}

\usepackage{hyperref}
\usepackage{amssymb}
\usepackage[bbgreekl]{mathbbol}
\usepackage{graphicx}
\usepackage{ifthen}
\usepackage{pict2e}
\usepackage{xargs}
\usepackage{xspace}
\usepackage{xcolor}
\usepackage{pgf,tikz}
\usepackage{pgfplots}
\usepackage{environ}
\usepackage{caption}
\usetikzlibrary{arrows}
\usepackage{enumitem}
\usepackage{xifthen}
\newcounter{dummy}
\makeatletter
\newcommand\myitem[1][]{\item[#1]\refstepcounter{dummy}\def\@currentlabel{#1}}
\makeatother

\makeatletter
\newsavebox{\measure@tikzpicture}
\NewEnviron{scaletikzpicturetowidth}[1]{%
	\def\tikz@width{#1}%
	\begin{lrbox}{\measure@tikzpicture}%
		\BODY
	\end{lrbox}%
	\pgfmathparse{#1/\wd\measure@tikzpicture}%
	\BODY
}
\makeatother

\DeclareSymbolFontAlphabet{\mathbb}{AMSb}

\newcommand{\thistheoremname}{}
\newtheorem*{genericthm*}{\thistheoremname}
\newenvironment{namedthm*}[1]
{\renewcommand{\thistheoremname}{#1}%
	\begin{genericthm*}}
	{\end{genericthm*}}

\newcommand{\Bairespace}[1][]{
	\ifthenelse{\equal{#1}{}}{\functions{\N}{\N}}{\functions{#1}{\N}}
}
\newcommand{\bbL}{\mathbb{L}}
\newcommand{\bbX}{\mathbb{X}}

\newcommand{\Cantorspace}[1][]{
	\ifthenelse{\equal{#1}{}}{\functions{\N}{2}}{\functions{#1}{2}}
}

\newcommandx{\concatenation}[2][1 = undefined, 2 = undefined]{
	\ifthenelse{\equal{#1}{undefined}}{{}\smallfrown}{
		\ifthenelse{\equal{#2}{undefined}}{\bigoplus #1}{\bigoplus_{#1} #2}
	}
}

\newcommandx{\functions}[3][3 =]{
	\ifthenelse{\equal{#3}{}}{#2^{#1}}{#2_{#3}^{#1}}
}

\newcommand{\Gzero}[1][]{
	\ifthenelse{\equal{#1}{}}
	{\mathbb{G}_0}
	{\mathbb{G}_{0,n}}
}
\newcommandx{\Hzero}[2][2 = undefined]{
	\ifthenelse{\equal{#2}{undefined}}
	{\mathbb{H}_{#1}}
	{\mathbb{H}_{#1, #2}}
}
\newcommandx{\intersection}[2][1 =, 2 =]{
	\ifthenelse{\equal{#1}{}}{\cap}{
		\ifthenelse{\equal{#2}{}}{\bigcap #1}{{\bigcap_{#1} #2}}
	}
}
\newcommand{\Lzero}[1][]{\ifthenelse{\equal{#1}{}}{\bbL_0}{L_{0, #1}}}
\newcommand{\Lzerospace}[1][]{\ifthenelse{\equal{#1}{}}{\bbX_0}{X_{0, #1}}}

\newcommand{\modulo}[1]{\ (\text{mod } 2)}
\newcommand{\N}{\mathbb{N}}

\newcommandx{\product}[2][1 =, 2 =]{
	\ifthenelse{\equal{#1}{}}{\times}{
		\ifthenelse{\equal{#2}{}}{\prod #1}{{\prod_{#1} #2}}
	}
}

\newcommandx{\sequence}[2][2 = undefined]{
	\ifthenelse{\equal{#2}{undefined}}{(#1)}{
		(#1)_{#2}
	}
}

\newcommandx{\set}[2][2 = undefined]{
	\ifthenelse{\equal{#2}{undefined}}{\{ #1 \}}{
		\{ #1 \suchthat #2 \}
	}
}
\newcommandx{\sets}[3][3 =]{
	\ifthenelse{\equal{#3}{}}{[#2]^{#1}}{[#2]^{#1}_{#3}}
}

\newcommand{\suchthat}{\mid}

\renewcommand{\restriction}[2]{#1 \upharpoonright #2}

\newcommandx{\union}[2][1 =, 2 =]{
	\ifthenelse{\equal{#1}{}}{\cup}{
		\ifthenelse{\equal{#2}{}}{\bigcup #1}{{\bigcup_{#1} #2}}
	}
}

\newtheorem{theorem}{Theorem}[section]
\newtheorem{lemma}[theorem]{Lemma}

\newtheorem{claim}[theorem]{Claim}

\newtheorem{proposition}[theorem]{Proposition}

\theoremstyle{definition}

\newtheorem{definition}[theorem]{Definition}

\numberwithin{equation}{section}

\newcommand{\bd}{\begin{definition}}
	\newcommand{\ed}{\end{definition}}

\DeclareMathOperator{\dist}{dist}
\DeclareMathOperator{\didistance}{didist}

\newcommand{\distance}[3]{\ifthenelse{\isempty{#3}}{\dist(#1,#2)}{\dist^{#3}(#1,#2)}}
\newcommand{\didist}[3]{\ifthenelse{\isempty{#3}}{\didistance(#1,#2)}{\didistance^{#3}(#1,#2)}}
\newcommand{\digraph}[3]{\ifthenelse{\equal{#1}{b}}{\mathbb{#2}_{#3}}
	{{#2}_{#3}}}
\newcommand{\linegraph}[3]{\ifthenelse{\equal{#1}{b}}{\mathbb{#2}_{#3}}
	{#2_{#3}}}

\newcommand{\underlyingspace}[3]{\ifthenelse{\equal{#1}{b}}{\mathbb{#2}_{#3}}
	{#2_{#3}}}
\newcommand{\distanceset}[2]{\ifthenelse{\isempty{#2}}{D(#1)}{D^{#2}(#1)}}

\newcommand{\concatt}{%
	\mathbin{\raisebox{1ex}{\scalebox{.7}{$\frown$}}}%
}

\title{Ramsey and Hypersmoothness}

\author[Z. Vidny\'{a}nszky]{Zolt\'{a}n Vidny\'{a}nszky}

	\address{
		Zolt\'{a}n Vidny\'{a}nszky \\
		California Institute of Technology\\ Department of Mathematics\\
1200 E California Blvd, Pasadena, CA 91125 \\ USA
	}
	\email{vidnyanz@caltech.edu}
	\urladdr{
		http://https://www.its.caltech.edu/~vidnyanz/
	}

\begin{document}

	\begin{abstract}
		Combining canonization results of Pr\"omel-Voigt, Mathias, and Soare, we provide a new, natural example of an $F_\sigma$ equivalence relation that is not hypersmooth. 
	\end{abstract}
		\maketitle
	
	Assume that $E$ is a countable Borel equivalence relation on the space $[\N]^\N$. It follows from the work of Mathias \cite{mathias} and Soare \cite{soare} that one can find an $x \in [\N]^\N$ such that $ \restriction{E}{[x]^\N}$ is hyperfinite (in fact, contained in $\mathbb{E}_0$, see \cite{zapletal}). In contrast, equivalence relations with uncountable classes can exhibit complicated behavior on every set of the form $[x]^\N$. 
	
	In this note, we give an example of an $F_\sigma$ equivalence relation on $[\N]^\N$ that is not hypersmooth (see \cite{kechris1997classification} for more information on such equivalence relations) and the reason for this is completely Ramsey-theoretic. The equivalence relation comes from a natural generalization of the shift-graph (see \cite{kechrissolecki,toden}) and it is generated by two continuous functions.
	
	\section{The construction}

	\begin{theorem}
		\label{t:mmain}
		Assume that $E$ is an equivalence relation on $[\N]^\N$ such that for every $x \in [\N]^\N$ 
		\begin{itemize}
			\item there exist almost disjoint $y,z \in [x]^\N$ such that $yEzEx$ and
			\item  $\restriction{E}{[x]^\N}$ has uncountably many equivalence classes.
		\end{itemize}
		 Then $E$ is not hypersmooth.
	\end{theorem}

		In order to establish this theorem, we will need the following statement.
	\begin{theorem}
			\label{t:main}
			Assume that $(f_n)_{n\in \N}:[\N]^{\N} \to (2^\N)^\N$ are Borel. There exist an $x \in [\N]^\N$ and a countable set $C$ such that for every $y,z \in [x]^\N$ almost disjoint and $n \in \N$ we have that $f_n(z)=f_n(y)$ implies $f_n(y) \in C$. 
	\end{theorem}
		
		The proof of this theorem is based on the results of Pr\"omel-Voigt \cite{promel}, Mathias \cite{mathias} and Soare \cite{soare} and postponed to the next section.
	\begin{proof}[Proof of Theorem \ref{t:mmain}]
		Assume the contrary, and let $\varphi$ be a Borel reduction from $E$ to $E_1$. Let $f_n=p^n \circ \varphi$ where $p((x_n)_{n \in \N})=(x_{n+1})_{n \in \N}$. Clearly, if $yEz$ then for some $n$ we have $f_n(x)=f_n(y)$. 
		
		Let $x$ and $C$ be as in Theorem \ref{t:main} for the sequence $(f_n)_{n \in \N}$. We claim that $\varphi([x]^\N)$ is contained in countably many $E_1$ equivalence classes, contradicting our second assumption. Indeed, take the countable set $C'=\{(x_n)_{n \in \N}:\exists m \ (\forall n<m \ x_n=0 \land (x_{m},x_{m+1},\dots) \in C)\}$. If $x' \in [x]^\N$ then by the first assumption, there are $y,z \in [x']^\N$ almost disjoint such that $yEzEx'$. Then, for some $n$ we have $f_n(y)=f_n(z) \in C$, so $\varphi(x') \in [C']_{E_1}$.
		\end{proof}
		We identify elements of $[\N]^\N$ with their increasing enumeration.  
		Let $S_0,S_1$ be maps on $[\N]^\N$ defined by \[S_0(\{n_0,n_1,n_2,\dots\})=\{n_1,n_2,n_3\dots\},\]
		\[S_1(\{n_0,n_1,\dots\})=\{n_1,n_3,n_5,\dots\}.\]	
	 	Denote by $G$ the graph determined by $S_0$ and $S_1$, and let $E$ be the connected component equivalence relation of $G$. 
		
		For $x,y \in [\N]^\N$ define $x \ll y$ if for each $m$ there is a $k$ such that $|x \cap (y_k,y_{k+1})|>m$. It is easy to check the following: 
		
		\begin{lemma}
			\label{l:dominates}
			Assume that $x \ll y$ and $\{x',x\} \in G$. Then $x'\ll y$.
		\end{lemma}
		
		Now we show that $E$ is an example of the phenomenon described in the main theorem.
	\begin{proposition}
		$E$ is not hypersmooth.
	\end{proposition}

	\begin{proof}
		We check the conditions of Theorem \ref{t:mmain}.
		First, for any $x'$ we have $S_1 \circ S_0(x') \cap S_1(x') =\emptyset$ and both $S_1 \circ S_0(x')$ and $S_1(x')$ are contained in $x'$.
		
		Second, assume that $x$ is given and $\{x_n:n \in \N\} \subset [x]^\N$ is arbitrary. It is sufficient to produce a $y \in [x]^\N$ such that $y$ is not $E$ equivalent to either of the $x_n$'s. Let $y \subset x$ be so that $x_n \ll y$ for all $n$. Then by Lemma \ref{l:dominates} and by $y \not \ll y$ we are done.
	\end{proof}
	
	The following proposition concludes our construction.
	\begin{proposition}
		$E$ is $F_\sigma$. 
	\end{proposition}
	\begin{proof}
		For $x,y \in [\N]^\N$ let $x\leq y$ stand for $\forall n \ (x_n\leq y_n)$. Note that for each $i<2$ we have $x \leq S_i(x)$. For $s \in 2^{<\N}$ let \[D_s=\{(x,y): x \leq S_{s(n-1)} \circ  \dots \circ S_{s(1)} \circ S_{s(0)}(y)\},\] 
		and let $D(y)=\bigcup_{s \in 2^{<\N}} \{x:(x,y) \in D_s\}$, that is, $D(y)$ is the collection of the points that are pointwise below some forward iterate of $y$.
		
		\begin{claim}
			\label{cl:dy}
			If $xGx'$ and $x \in D(y)$ then $x' \in D(y)$. Thus, $[y]_{E} \subseteq D(y)$ for all $y$.
		\end{claim} 
		\begin{proof}
			If $x'=S_i(x)$ for some $i<2$ and $(x,y) \in D_s$, then $(x',y) \in D_{s \concatt (i)}$. Otherwise, if $x=S_i(x')$, then $x'\leq S_i(x')=x$, so $x' \in D(y)$ as well.
		\end{proof}
		
		Let $(x,y) \in R_n$ iff $\exists k \leq n \ \exists x_0,\dots,x_{k-1} \ \exists s_0,\dots,s_{k-1} \in 2^{\leq n}$
		\[ \forall i<k \ (x_i G x_{i+1} \land (x_i,y) \in D_{s_i})\land x_0=x \land x_{k-1}=y.\]
		By Claim \ref{cl:dy}, $E= \bigcup_n R_n$. It follows from the continuity of the functions $(S_i)_{i<2}$ that each $D_s$ is closed and that for each $y$ the set $\{x:(x,y) \in D_s\}$ is compact. From this, it is straightforward to show that $R_n$ is closed for every $n \in \N$, which finishes our proof. 
	\end{proof}

	\section{A version of Pr\"omel-Voigt}
	
		Let us recall the canonization results of Pr\"omel and Voigt \cite{promel}. 
	
		\begin{theorem}
			[Pr\"omel-Voigt \cite{promel}]
			\label{t:prvoigt}
			Assume that $f:[\N]^\N \to 2^\N$ is a Borel function. There exist an $x \in [\N]^\N$ and a function $\Gamma:[x]^\N \to [\N]^{\leq \N}$ with the following properties:
		\begin{enumerate}	
			\item for every $y \in [x]^\N$ we have $\Gamma(y) \subseteq y$, 
				\item \label{p:canon} for every $y,z \in [x]^\N$ we have 
			\[f(y)=f(z) \iff \Gamma(y)=\Gamma(z)\]
			\item \label{p:finorinf} either for every $y \in [x]^\N$ we have that $\Gamma(y)$ is finite or for every $y \in [x]^\N$ it is infinite,
			\item \label{p:initial} for every $y,z \in [x]^\N$ we have that $\Gamma(y)$ is not a proper initial segment of $\Gamma(z)$,
			\item $\Gamma$ is continuous (where the topology on $[\N]^{ \leq \N}$ comes from its ideintification with $2^\N$).
		\end{enumerate}
		\end{theorem}
		
		For $x \in [\N]^\N$ and $f:[\N]^\N \to X$ define \[C_{f,x}=\{f(x'):x' \mathbb{E}_0x\}.\]
	
		\begin{lemma}
			\label{l:const}
			Let $x$, $\Gamma$ and $f$ be as above. If for some $y \in [x]^\N$ the set $\Gamma(y)$ is finite, then for all $x' \in [x]^\N$ we have $f(y) \in C_{f,x'}$. 
		\end{lemma}
		\begin{proof}
			Let $x' \in [x]^\N$ be arbitrary. By the continuity of $\Gamma$ we can find some $x'' \mathbb{E}_0x'$ with $x'' \in [x]^\N$ such that $\Gamma(y)$ is an initial segment of $\Gamma(x'')$. By \ref{p:initial} this implies that $\Gamma(x'')=\Gamma(y)$, so $f(y)=f(x'') \in C_{f,x'}$.  
		\end{proof}
		
		Let us point out that the preceeding argument is the only thing what one needs to prove Theorem \ref{t:main}, in the case of a single function. However, to treat countably many functions, we have to prove an analogous statement for sets of the form $[s,x]$, where $s \in \N^{<\N}$ and $x \in [\N]^\N$. This requires some work and the utilization of the result of Mathias and Soare mentioned in the introduction.

		If $\Gamma$ is a function as above, define a partial function $\Gamma^*:2^\N \to [x]^\N$ by
		\[\Gamma^*(w)=z\iff \exists y \in [x]^\N \ f(y)=w \text{ and } \Gamma(y)=z.\]
		The following is obvious from the definition.
		\begin{lemma} 
			\label{l:ongamma} 
			\begin{itemize}
				\item $\Gamma^*$ is relatively Borel on the (possibly proper) analytic set $f([x]^\N)$; in particular, it is Souslin-measurable.
				\item For every $y \in [x]^\N$ we have \[\Gamma^*(f(y))=\Gamma(y).\]
			\end{itemize}
			
		\end{lemma}

		We will need the following theorem (see also \cite{zapletal}).
	
		\begin{theorem}[Mathias \cite{mathias}, Soare \cite{soare}]
			\label{t:canonize}
			Let $(g_n)_{n \in \N}:[\N]^\N \to [\N]^\N$ be a collection of Souslin-measurable functions. There exists an $x \in [\N]^\N$ such that for every $y \in [x]^\N$ and for every $n \in \N$ we have that $g_n(y) \in [x]^\N$ implies that $g_n(y) \setminus y$ is finite.    
		\end{theorem}
	
		Note that Soare only establishes the above theorem for Borel functions (and the proof of Kanovei-Sabok-Zapletal is also only for such), however, as by Silver's theorem \cite{silver1970every} every Souslin-measurable set has the Ramsey property, the proof extends to the class of Souslin-measurable maps without further modification.

	The key ingredient to prove Theorem \ref{t:mmain} is the following. For a function $f$ and an $x \in [\N]^\N$ define $C_{f,x}=\{f(x'):x'\mathbb{E}_0x\}$.

	\begin{lemma}
		\label{l:main}
		Let $s \in \N^{<\N}$ and $x \in [\N]^\N$ with $\max s <\min x$ be arbitrary and $f:[\N]^\N \to 2^\N$ be a Borel function. There exists an $x'\in [x]^\N$ such that whenever $y,z \in [s \cup x']^\N$ are almost disjoint and $f(y)=f(z)$ then $f(y) \in C_{f,x''}$ for each $x'' \in [x']^\N$.
	\end{lemma}

	\begin{proof}
		For $t \subseteq s$ and $y \in [x]^\N$ define a function $f^t$ by \[f^t=f(t \cup y).\]
		
		By applying Theorem \ref{t:prvoigt} $2^{|s|}$ many times, we can find an $x^* \in [x]^\N$ and functions $\Gamma_{t},\Gamma^*_t$ with the properties as in the theorem with respect to the function $f^t$ on $[x^*]^\N$. For $r,t \subseteq s$ let $g^{r,t}$ be the function $\Gamma^*_t \circ f^r$. Note that $g^{r,t}$ is Souslin-measurable. By Theorem \ref{t:canonize} we can find an $x' \in [x^*]^\N$ such that whenever $y \in [x']^\N$ and $g^{r,t}(y) \in [x']^\N$ for some $r,t \subseteq s$ then $g^{r,t}(y) \setminus y$ is finite. 
		
		We claim that $x'$ has the desired property. Let $y,z \in [s, x']^\N$ be almost disjoint, and assume that $f(y)=f(z)$, let $r=y \cap s$ and $t=z \cap s$. Using Lemma \ref{l:const} it suffices to show that $\Gamma_r(y \cap x')$ is finite: indeed, in this case $f(y)=f^r(y \cap x') \in C_{f^r,x''} \subseteq C_{f,x''}$ for each $x'' \in [x^*]^\N$, in particular $f(y) \in C_{f,x'}$. Now, 
		 \[g^{t,r}(z \cap x')=\Gamma^*_r(f^t(z\cap x'))=\Gamma^*_r(f(t \cup z))=\Gamma^*_r(f(z))=\]\[\Gamma^*_r(f(y))=\Gamma^*_r(f(r \cup y))=\Gamma^*_r(f^r(y))=\Gamma^*_r(f^r(y \cap x'))=\]\[\Gamma_r(y \cap x') \subseteq y \cap x'.\]
		 So, as $z \cap x', g^{t,r}(z \cap x') \in [x']^\N$, we must have that \[ g^{t,r}(z \cap x') \setminus (z \cap x')=\Gamma_r(y \cap x') \setminus (z \cap x') \] is finite. Therefore, $\Gamma_r(y \cap x')$ is almost contained in both $y \cap x'$ and $z \cap x'$, which show that this set is finite. 
	 \end{proof}

	Finally, we prove Theorem \ref{t:main} using a standard fusion argument.
	
	\begin{proof}[Proof of Theorem \ref{t:main}]
 	Define sequences $s_0 \subset s_1 \subset \dots$ and $x_0 \supseteq x_1 \supseteq \dots $ of elements in $[\N]^{<\N}$ and $[\N]^{\N}$ respectively, such that for every $n \in \N$
 	\begin{itemize}
 		\item $\max s_n<\min x_{n}$,
 		\item $s_{n+1}=s_n \cup \{\min x_{n}\}$,
 		\item for every $y,z \in [s_{n} \cup x_n]^\N$ almost disjoint, if $f_n(y)=f_n(z)$ then $f_n(y) \in C_{f_n,x_n}$
 	\end{itemize}
 	by induction on $n$. Indeed, let $s_{-1}=\emptyset$ and $x_{-1}=\N$, given $s_{n-1},x_{n-1}$, let $s_n =s_{n-1}\cup \{\min x_{n-1}\}$. Using Lemma \ref{l:main} for every $s \subseteq s_n$, we can find an $x_{n} \subseteq x_{n-1}$ that satisfies all the above properties. 
 	
 	Now let $x=\bigcup_{n \in \N} s_n$, $C=\bigcup_n C_{f_n,x_n}$, and assume that $y,z \in [x]^\N$ are almost disjoint and $f_n(y)=f_n(z)$. Then $y,z \in [s_{n} \cup x_n]^\N$, and hence $f_n(y) \in C_{f_n,x_n} \subseteq C$, by the third property of the induction. 
 \end{proof}

	\textbf{Acknowledgments.} We would like to thank Jan Greb\'ik, Alexander Kechris,  Adrian Mathias, and Stevo Todor\v{c}evi\'c for valuable comments and suggestions.


%

	\bibliographystyle{abbrv}
	\bibliography{bbl.bib}
	
\end{document}